\documentclass[11pt,reqno]{amsart}
\usepackage{amsmath, amssymb, amsthm}
\usepackage{url}
\usepackage{hyperref}
\usepackage{tikz}
\usepackage{ytableau}
\usepackage[utf8]{inputenc}
\usepackage{xcolor}
\usepackage{tikz-cd}
\usepackage{hyperref}
\usepackage{nccmath}

\usepackage{youngtab}

\usetikzlibrary{decorations.pathreplacing}

\usetikzlibrary{decorations.pathreplacing}

\renewcommand{\(}{\left\(}
\renewcommand{\)}{\right\)}
\renewcommand{\[}{\left\[}
\renewcommand{\]}{\right\]}

\numberwithin{equation}{section}
\theoremstyle{plain}
\newtheorem{theorem}{Theorem}[section]
\newtheorem{lemma}[theorem]{Lemma}

\newtheorem{corollary}[theorem]{Corollary}

\usepackage{mathrsfs}

\makeatletter
\def\proof{\@ifnextchar[{\@oproof}{\@nproof}}
\def\@oproof[#1][#2]{\trivlist\item[\hskip\labelsep\textit{#2 Proof of\
		#1.}~]\ignorespaces}
\def\@nproof{\trivlist\item[\hskip\labelsep\textit{Proof.}~]\ignorespaces}

\makeatother

\makeatletter
\@namedef{subjclassname@2020}{%
	\textup{2020} Mathematics Subject Classification}
\makeatother

\begin{document}
	
	\title[]{Asymptotics and inequalities for the distinct
partition function}
	
		
	\maketitle
	
\vskip 0.2cm

\begin{center}
{Gargi Mukherjee}$^{1}$, {Helen W.J. Zhang}$^{2}$ and
  {Ying Zhong}$^{2}$ \vskip 2mm
$^{1}$School of Mathematical Sciences\\[2pt]
National Institute of Science Education and Research, Bhubaneswar, An OCC of Homi Bhabha National Institute,  Via- Jatni, Khurda, Odisha- 752050, India\\[5pt]

$^{2}$School of Mathematics\\[2pt]
Hunan University, Changsha 410082, P.R. China\\[5pt]

Email: gargi@niser.ac.in, \quad helenzhang@hnu.edu.cn, \quad YingZhong@hnu.edu.cn
\end{center}

	\begin{abstract}
In this paper, we give explicit error bounds for the asymptotic
expansion of the shifted distinct partition function $q(n +s)$ for any nonnegative integer $s$. Then based on this refined asymptotic formula, we give the exact thresholds of $n$ for the inequalities derived from the invariants of the quartic binary form, the double Tur\'{a}n inequalities, the Laguerre inequalities and their corresponding companion versions.

\vskip 6pt

\\ [7pt]
	\end{abstract}
	
\section{Introduction}\label{intro}
The objective of this paper is to derive an asymptotic expansion for the distinct partition function and further show that
this function satisfies the inequalities derived from the invariants of the quartic binary form, the double Tur\'{a}n inequalities, the Laguerre inequalities and their companion versions.

A sequence $\{\alpha_n\}_{n\geq0}$ is said to satisfy the double Tur\'{a}n inequalities if for $n\geq2$,
\begin{equation*}
\left(\alpha_n^2- \alpha_{n-1}\alpha_{n+1}\right)^2
\geq \left(\alpha_{n-1}^2- \alpha_{n-2}\alpha_{n}\right)
\left(\alpha_{n+1}^2- \alpha_{n}\alpha_{n+2}\right).
\end{equation*}
Additionally, the sequence satisfies Laguerre inequalities of order $m$ if
\begin{equation*}
\frac 12\sum_{k=0}^{2m}(-1)^{k+m} {{2m}\choose k}\alpha_{n+k}\alpha_{2m-k+n} \geq 0.
\end{equation*}
It should be noted that the study of the Tur\'{a}n-type and Laguerre inequalities is closely related to the Laguerre-P\'{o}lya class and the Riemann hypothesis, see, for example,  \cite{Csordas-Varga-1990,Dimitrov-1998,Dimitrov-Lucas-2011,
	Schur-Polya-1914,Szego-1948}. Consequently, this area has received widespread attention. Many scholars have been investigating the Tur\'{a}n-type and their companion inequalities \cite{Chen-Jia-Wang-2019,DeSalvo-Pak-2015,Griffin-Ono-Rolen-Zagier-2019}, as well as the Laguerre and corresponding companion inequalities \cite{Wagner-2022,Wang-Yang-2022,Wang-Yang-2024} for different
kind of sequences.

A comprehensive study on inequalities arising from invariants of a binary form was studied by Chen \cite{Chen-2017}. For the background
on the theory of invariants, see, for example, Hilbert \cite{Hilbert-1993}, Kung and Rota
\cite{Kung-Rota} and Sturmfels \cite{Sturmfels}. A binary form $P_d(x,y)$ of degree $d$ is a homogeneous polynomial of degree $d$ in two variables $x$ and $y$, defined by
$$P_d(x,y):=\sum_{i = 0}^{d}\binom{d}{i}a_ix^{i}y^{d - i},$$
where the coefficients $a_i$ are complex numbers. Restricting $a_i$ to be real numbers, the binary form $P_d(x,y)$ is transformed into a new binary form
$$Q_d(\bar{x},\bar{y})=\sum_{i = 0}^{d}{d \choose i}c_i\bar{x}^{i}\bar{y}^{d - i}$$
under the action of an invertible complex matrix $M = \begin{pmatrix}m_{11}&m_{12}\\m_{21}&m_{22}\end{pmatrix}$ as follows:
$$\begin{pmatrix}x\\y\end{pmatrix}=M\begin{pmatrix}\bar{x}
\\\bar{y}\end{pmatrix}.$$
The transformed coefficients $c_i$ are polynomials in $a_i$ and $m_{ij}$. For nonnegative
integer $k$, a polynomial $I(a_0,a_1,\cdots,a_d)$ in the coefficients $(a_i)_{0\leq i\leq d}$ is called an invariant of index of $k$ of the binary form $P_d(x,y)$ if for any invertible matrix $M$,
$$I(\bar{a}_0,\bar{a}_1,\cdots,\bar{a}_d)=(\text{det }M)^{k}I(a_0,a_1,\cdots,a_d).$$
We observe that
three invariants of the quartic binary form
$$P_4(x,y)=a_4x^{4}+4a_3x^{3}y + 6a_2x^{2}y^{2}+4a_1xy^{3}+a_0y^{4}$$
are of the following form
\begin{align*}
A(a_0,a_1,a_2,a_3,a_4)&=a_0a_4-4a_1a_3 + 3a_2^{2},
\\
B(a_0,a_1,a_2,a_3,a_4)&=-a_0a_2a_4+a_2^{3}+a_0a_3^{2}
+a_1^{2}a_4-2a_1a_2a_3,
\\
I(a_0,a_1,a_2,a_3,a_4)&=A(a_0,a_1,a_2,a_3,a_4)^3-27B(a_0,a_1,a_2,a_3,a_4)^2.
\end{align*}
Chen \cite{Chen-2017} conjectured that both the partition function and the spt-function satisfy the inequalities derived from the invariants of the quartic binary form and the associated companion inequalities. For
the definition of spt-function, we
refer to Andrews \cite{Andrews-2008}. Chen's conjectures related to $A(a_0,a_1,a_2,a_3,a_4)$ and $B(a_0,a_1,a_2,a_3,a_4)$ have recently been proved by Banerjee \cite{Banerjee-2025}, Jia
and Wang \cite{Jia-Wang-2020} and Wang and Yang \cite{Wang-Yang-2022}. 	

Recently, inequalities for the distinct partition function have also received significant attention from many researchers. Recall that a distinct partition of integer $n$ is
a partition of $n$ into distinct parts. For example, there are eight distinct partitions of $9$:
\begin{equation*}
  (9), (8, 1), (7, 2), (6, 3), (6, 2, 1), (5, 4), (5, 3, 1), (4, 3, 2).
\end{equation*}
Let $q(n)$ denote the number of
distinct partitions of $n$.
Hagis \cite{Hagis-1963} and Hua \cite{Hua-1942} established a Hardy-Ramanujan-Rademacher type formula
for $q(n)$ by applying the Hardy-Ramanujan circle method. Applying this formula, Craig and Pun \cite{Craig-Pun-2021} showed that $q(n)$ satisfies the Tur\'{a}n inequalities of any order for sufficiently large $n$ by using  a general result of Griffin, Ono, Rolen, and Zagier \cite{Griffin-Ono-Rolen-Zagier-2019}. They also conjectured that $q(n)$ is log-concave for $n\geq33$ and satisfies the third-order Tur\'{a}n inequalities for $n\geq121$.
Instead of using the
Hardy-Ramanujan-Rademacher type formula for $q(n)$ due to Hagis \cite{Hagis-1963} and Hua \cite{Hua-1942}, Dong and Ji \cite{Dong-Ji-2024} used Chern's asymptotic formulas for $\eta$-quotients \cite{Chern-2019} to obtain an asymptotic formula for $q(n)$
with an effective bound on the error term. Based on this asymptotic formula, they proved Craig and Pun's \cite{Craig-Pun-2021} conjectures.
But they couldn't prove the inequalities pertaining to the invariants of the quartic binary form for $q(n)$ and stated them as conjectures \cite[Conjecture 6.1]{Dong-Ji-2024}.
In this paper, we derive an asymptotic
expansion of the shifted distinct partition function $q(n +s)$ for any nonnegative integer $s$. Subsequently, we prove the first two conjectures of Dong and Ji \cite{Dong-Ji-2024} and establish their associated companion inequalities.

To state our asymptotic expansion, we need to introduce some definitions. Here and throughout, we use the notation $f(n)=O_{\leq C}(g(n))$ to mean $|f(x)|\leq C\cdot g(x)$ for a positive function $g$ and the domain for $x$ will be specified accordingly.
For $(s,N) \in \mathbb{N}_0\times\mathbb{N}$, we define
\begingroup
\allowdisplaybreaks
\begin{align}
\label{n(N,s)}
n(N,s)&:=\max\left\{206,\frac{\frac{72}{\pi^2}N^2_0(N+2)-1}{24}, \left \lceil \frac{2(24s+1)}{3} \right \rceil\right\},\\\nonumber
Er_N(s)&:=|a_N(1)|\left(\frac{\pi\cdot2^{N-1}}{\sqrt{3}}
\sqrt{\frac{24s+1}{24}}+A(s) \left(1+\frac{2^{N+1}}{3}\right) \right)\Bigl(\frac{24s+1}{24}\Bigr)^{\frac{N+1}{2}}\\\label{finalerrordef} &\hspace{-1.5 cm}+\left(1+\frac{\pi}{2\sqrt{3}}\frac{24s+1}{24}+ \frac{A(s)}{12} \right) Er_6(N,s)+2 |a_N(1)| Er_4(N,s)+\frac{Er_4(N,s)Er_6(N,s)}{n(N,s)^{\frac{N+1}{2}}},
\end{align}
\endgroup
where $a_m(v), Er_3(N,s), Er_4(N,s), Er_6(N,s), \text{and}\ A(s)$ are as defined in \eqref{a}, \eqref{eqn3}, \eqref{Er_{4}(N,s)}, \eqref{eqn6}, and \eqref{eqn8} respectively.

\begin{theorem}\label{MainTh}
	For all $(s, N) \in \mathbb{N}_0\times\mathbb{N}$ and $n \ge n(N,s)$, we have
	\begin{equation*}
	q(n+s)=\frac{e^{\pi\sqrt{\frac{n}{3}}}}{4\cdot3^{1/4}n^{3/4}}\Biggl(\sum_{m=0}^{N}\frac{\widehat{B}_m(s)}{n^{\frac m2}}+O_{\le Er_N(s)}\Bigl(n^{-\frac{N+1}{2}}\Bigr)\Biggr),
	\end{equation*}
	where the coefficient sequence $\left(\widehat{B}_m(s)\right)_{m\ge 0}$ is given explicitly in \eqref{finalcoeffdef}.
\end{theorem}

In light of Theorem \ref{MainTh}, we obtain the following inequalities by taking appropriate values of $N$.

\begin{theorem}\label{A-thm}
	For $n\geq 230$, we have
	\begin{equation*}
	A\left(q(n-1),q(n), q(n+1), q(n+2), q(n+3)\right)>0.
	\end{equation*}
\end{theorem}

\begin{theorem}\label{A-r-thm}
	For $n\geq 279$, we have
	\begin{equation*}
	4\left(1+\frac{\pi^2}{32n^3}\right)q(n)q(n+2)>q(n-1)q(n+3)+3q(n+1)^2.
	\end{equation*}
\end{theorem}

\begin{theorem}\label{B-thm}
	For $n\geq 272$, we have
	\begin{equation*}
	B\left(q(n-1),q(n), q(n+1), q(n+2), q(n+3)\right)>0.
	\end{equation*}
\end{theorem}

\begin{theorem}\label{B-r-thm}
	For $n\geq 309$, we have
	\begin{align*}
	&\left(1+\frac{\pi^3}{288\sqrt{3}n^{\frac{9}{2}}}\right)
	\left(2q(n)q(n+1)q(n+2)+q(n-1)q(n+1)q(n+3)\right)
	\\
	&\hspace{5cm}>q(n+1)^3+q(n-1)q(n+2)^2+q(n)^2q(n+3).
	\end{align*}
\end{theorem}

\begin{theorem}\label{d-turan}
	For $n\geq 273$, $q(n)$ satisfies the double Tur\'{a}n inequalities, i.e.,
	\begin{align*}
	\left({q(n)}^2-q(n-1)q(n+1)\right)^2
	>\left({q(n-1)}^2-q(n-2)q(n)\right)\left({q(n+1)}^2-q(n)q(n+2)\right).
	\end{align*}
\end{theorem}

\begin{theorem}\label{c-d-turan}
	For $n\geq 346$, $q(n)$ satisfies the companion double Tur\'{a}n inequalities, i.e.,
	\begin{align*}
	&\left({q(n)}^2-q(n-1)q(n+1)\right)^2
	\\
	&\qquad\qquad<\left({q(n-1)}^2-q(n-2)q(n)\right)
	\left({q(n+1)}^2-q(n)q(n+2)\right)\left(1+\frac{\pi}{2\sqrt{3}n^{\frac 32}}\right).
	\end{align*}
\end{theorem}

More recently, similar to the general result of Griffin, Ono, Rolen, and Zagier \cite{Griffin-Ono-Rolen-Zagier-2019} for the Tur\'an inequalities of any order, Wang and Yang \cite{Wang-Yang-2024} gave a criteria for the Laguerre inequalities of any order. It is worth noting that if a sequence satisfies the conditions of Griffin, Ono, Rolen, and Zagier \cite{Griffin-Ono-Rolen-Zagier-2019}, then it also satisfies
the conditions of Wang and Yang \cite{Wang-Yang-2024}.
Craig and Pun \cite{Craig-Pun-2021} have showed that $q(n)$ satisfies the criteria of Griffin, Ono, Rolen, and Zagier \cite{Griffin-Ono-Rolen-Zagier-2019}. This directly implies the following result.

\begin{corollary}
	For sufficiently large $n$, $q(n)$ satisfies the Laguerre inequalities of any order.
\end{corollary}

With the aid of Theorem \ref{MainTh}, we determine the threshold
for the Laguerre inequalities of order 3 and their companion inequalities.
In fact, the Laguerre inequalities of order $2$ is equivalent to the inequality $A(a_0, a_1, a_2, a_3, a_4)\geq0$. The findings regarding the third-order Laguerre inequalities are presented as follows.

\begin{theorem}\label{3-laguerre}
For $n\geq 651$, $q(n)$ satisfies the Laguerre inequalities of order $3$, i.e.,
\begin{align*}
 10{q(n+3)}^2+6q(n+1)q(n+5)>15q(n+2)q(n+4)+q(n)q(n+6).
\end{align*}
\end{theorem}

\begin{theorem}\label{c-3-laguerre}
For $n\geq 715$, $q(n)$ satisfies the companion Laguerre inequalities of order $3$, i.e.,
\begin{align*}
 10{q(n+3)}^2&+6q(n+1)q(n+5)
 \\
 &<\left(15q(n+2)q(n+4)+q(n)q(n+6)\right)
 \left(1+\frac{5\pi^3}{256\sqrt{3}n^{\frac{9}{2}}}\right).
\end{align*}
\end{theorem}

The paper is organized as follows. In Section \ref{sec-set}, we establish the upper and lower bounds for $q(n)$, which in turn furnish an asymptotic expansion of $q(n+s)$. In Section \ref{sec-pre}, we calculate each term of the expansion separately by Taylor's theorem. These enable us to prove Theorem \ref{MainTh} in Section \ref{sec-asym-thm}. Finally, in Section \ref{sec-ineq}, we prove all the inequalities by applying Theorem \ref{MainTh}.

	\section{Set up}\label{sec-set}

In this section, by giving an upper bound and a lower bound for $q(n)$, and along with the estimations of the first modified Bessel function of the first kind,we shall obtain a preliminary asymptotic expansion of $q(n+s)$.

	Define \begin{equation}\label{N0(m)}
N_0(m):=
\begin{cases}
1, &\quad \text{if}\ m=1,\\
4m \log m -3m\log \log m, & \quad \text{if}\ m\ge2,
\end{cases}
\end{equation}
and \begin{equation}\label{M(n)}
M(n):=\frac{\sqrt{2}\pi^2}{12 \nu(n)}I_1(\nu(n)),
\end{equation}
where $I_1(s)$ is the first modified Bessel function of the first kind and
$$\nu(n):=\frac{\pi\sqrt{24n+1}}{6\sqrt{2}}.$$

\begin{lemma}\label{q-asym-lem}
	For $\nu(n)\ge \max\{26, N_0(m+1)\}$, we have
	\begin{equation}\label{q(n)first}
	M(n)\Biggl(1-\frac{4}{\nu(n)^m}\Biggr)\le q(n)\le M(n)\Biggl(1+\frac{4}{\nu(n)^m}\Biggr).
	\end{equation}
\end{lemma}
\begin{proof}
	From \cite[Theorem 1.2]{Dong-Ji-2024}, we know that for $\nu(n)\geq 21$, \begin{equation*}
	q(n)=\frac{\sqrt{2}\pi^2}{12 \nu(n)}I_1(\nu(n))+R(n),
	\end{equation*}
	where
	\begin{equation*}
	|R(n)|\leq \frac{\sqrt{3}\pi^{\frac{3}{2}}}{6\nu^{\frac{1}{2}}}
	\exp\left(\frac{\nu(n)}{3}\right).
	\end{equation*}
	As a result, let
	\begin{equation*}
	G(n):=\sqrt{\frac{6\nu(n)}{\pi}}\cdot \frac{\exp\left(\frac{\nu(n)}{3}\right)}{I_1(\nu(n))},
	\end{equation*}
	we have
	\begin{equation*}
	M(n)\left(1-G(n)\right)\leq q(n)\leq  M(n)\left(1+G(n)\right).
	\end{equation*}
	Hence, to verify \eqref{q(n)first}, it suffices to show that for $\nu(n)\geq \max\{26,N_0(m+1)\}$,
	\begin{equation}\label{b-G}
	G(n)\leq \frac{4}{\nu(n)^{m}}.
	\end{equation}
	Dong and Ji \cite[Eq.(3.16)]{Dong-Ji-2024} gave that for $\nu(n)\geq 26$,
	\begin{equation}\label{G}
	G(n)\leq 2\sqrt{3}\nu(n)\left(1+\frac{1}{\nu(n)}\right)
	\exp\left(-\frac{2\nu(n)}{3}\right).
	\end{equation}
	It's easy to check that for $\nu(n)\geq 7$,
	\begin{equation}\label{G-1}
	2\sqrt{3}\nu(n)\left(1+\frac{1}{\nu(n)}\right)\leq 4\nu(n).
	\end{equation}
	Analogous to the proof of \cite[Lemma 2.1]{Mukherjee-Zhang-Zhong-2023}, for $x\geq N_0(m)$, we have
	\begin{equation*}
	\exp\left(-\frac{2x}{3}\right)<x^{-m}.
	\end{equation*}
	It follows that for $\nu(n)\geq N_0(m+1)$,
	\begin{equation}\label{G-2}
	\exp\left(-\frac{2\nu(n)}{3}\right)<\frac{1}{\nu(n)^{m+1}}.
	\end{equation}
	Plugging \eqref{G-1} and \eqref{G-2} into \eqref{G}, we get \eqref{b-G} and thus complete the proof.
\end{proof}

In view of \cite[Theorem 3.9]{B}, for $N\geq 1$, we have
\begin{equation*}
\left|\frac{\sqrt{2\pi x}}{e^x}I_1(x)-\sum_{m=0}^N \frac{(-1)^m a_m(1)}{x^m}\right|<\frac{Er_{N,1}}{x^{N+1}},
\end{equation*}
where
\begin{align}
\label{a}
a_m(v)&:=\frac{{v-\frac{1}{2}\choose m}\left(v+\frac{1}{2}\right)_{m}}{2^{m}},
\\\label{eqn1}
Er_{N,1}&:=\frac{3^{\frac{N+1}{2}}}{\pi^{N+1}}\left(\frac{1+\frac{9}{\log(N+1)}+\frac{9}{N+2}}{\sqrt{2\pi}}
+\frac{\sqrt{2}+\left(N+\frac{5}{2}\right)
	^{-\frac{1}{2}}}{\log(N+1)}\right)|a_{N+1}(1)|.
\end{align}
Combining with Lemma \ref{q-asym-lem}, for $(s, N) \in \mathbb{N}_0\times\mathbb{N}$ and $\nu(n)\ge \max\{26,\ N_0(N+2)\}$, we have
\begingroup
\allowdisplaybreaks
\begin{eqnarray}\label{q(n)second}	q(n+s)&=&M(n+s)\left(1+O_{\le4}\left(\frac{3^{\frac{N+1}{2}}}{\pi^{N+1}}
n^{-\frac{N+1}{2}}\right)\right) \nonumber \\ &=& \frac{e^{\pi\sqrt{\frac{n}{3}}}}{4\cdot3^{1/4}n^{3/4}}
\frac{e^{\pi\sqrt{\frac{n}{3}}\left(\sqrt{1+\frac{24s+1}{24n}}-1\right)}}
{\Bigl(1+\frac{24s+1}{24n}\Bigr)^{3/4}}\Biggl(\sum_{m=0}^{N}\frac{(-1)^m a_m(1)}{(\frac{\pi}{\sqrt{3}})^m n^{\frac{m}{2}}}\left(1+\frac{24s+1}{24n}\right)^{-\frac{m}{2}} \nonumber \\ & & \hspace{2.5 cm}+O_{Er_{N,1}}\left(
n^{-\frac{N+1}{2}}\right)\Biggr)\left(1+O_{\le4}\left(\frac{3^{\frac{N+1}{2}}}{\pi^{N+1}}
n^{-\frac{N+1}{2}}\right)\right).
\end{eqnarray}
\endgroup

\section{Preliminary and preparatory lemmas}\label{sec-pre}

In this section, through the application of Taylor's theorem, we separately compute each term of the preliminary asymptotic expansion of
$q(n+s)$ given in \eqref{q(n)second}, and these estimations are indispensable for the proof of Theorem \ref{MainTh}. To this end, we need the following result.

\begin{lemma}\cite[Lemma 3.3]{BPRS}\label{lemma2.3}
	For $r,m\in {\mathbb{N}}_0$ with $r<2m$, we have
	\begin{equation*}
	\sum_{s=0}^{r}(-1)^s\binom{r}{s}\binom{\frac{s}{2}}{m}=
	\begin{cases}
	1, &\quad \text{if}\ r=m=0,\\
	(-1)^m \frac{r \cdot 2^r}{m \cdot 2^{2m}} \binom{2m-r-1}{m-r}, & \quad \text{otherwise}.
	\end{cases}
	\end{equation*}
\end{lemma}
Following the work done in \cite{M}, we have the following preparatoty Lemmas \ref{exp} - \ref{besselER}. Before that,
for $k\in \mathbb{N}_0$, define
\begin{equation*}
B_{2k}(s):=
\begin{cases}
1, &\quad \text{if}\ k=0,\\
\frac{(\frac{24s+1}{24})^k\left(\frac{1}{2}-k\right)_{k+1}}{k} \displaystyle\sum_{\ell_1=1}^{k}\frac{ (-k)_{\ell_1}}{(k+\ell_1)!}\frac{\Bigl(\pi \sqrt{\frac{24s+1}{72}}\Bigr)^{2\ell_1}}{(2\ell_1-1)!}, &\quad \text{if}\ k\ge 1,
\end{cases}
\end{equation*}
and
\begin{equation}\label{expcoeffdef2}
B_{2k+1}(s):=\frac{\pi}{\sqrt{3}}\left(\frac{24s+1}{24}\right)^{k+1}\left(\frac{1}{2}-k \right)_{k+1} \sum_{\ell_1=0}^{k}\frac{(-k)_{\ell_1}}{(\ell_1+k+1)!}\frac{\Bigl(\pi \sqrt{\frac{24s+1}{72}}\Bigr)^{2\ell_1}}{(2\ell_1)!}.
\end{equation}
Moreover, we let
\begin{equation}\label{eqn2}
Er_2(N,s):=\frac 43  \sqrt{\frac{2\pi}{3}}N^{-\frac{3}{2}}\Bigl(\frac{24s+1}{24}\Bigr)^{\frac{N+2}{2}} \cosh\Bigl(\pi \sqrt{\frac{24s+1}{72}}\Bigr).
\end{equation}

\begin{lemma}\label{exp}
	For $N \geq 1$ and $n \geq \lceil \frac{2(24s+1)}{3} \rceil$, we have
	\begin{equation*}
	e^{\pi\sqrt{\frac{n}{3}}\Bigl(\sqrt{1+\frac{24s+1}{24n}}-1\Bigr)}= \sum_{k=0}^{N}\frac{B_k(s)}{n^{\frac{k}{2}}} + O_{\leq Er_2(N,s)}\left(n^{-\frac{N+1}{2}}\right),
	\end{equation*}
\end{lemma}
\begin{proof}
By Taylor expansion, we have
\begingroup
\allowdisplaybreaks
\begin{eqnarray}\label{expeqnmain}
e^{\pi\sqrt{\frac{n}{3}}\Bigl(\sqrt{1+\frac{24s+1}{24n}}-1\Bigr)}&=&\sum_{r_1=0}^{\infty}\frac{\Bigl(\frac{\pi\sqrt{n}}{3}\Bigr)^{r_1}}{{r_1}!}\Biggl(\sqrt{1+\frac{24s+1}{24n}}-1\Biggr)^{r_1}
\nonumber \\ &=& \sum_{r_1=0}^{\infty}\frac{\Bigl(\frac{\pi\sqrt{n}}{3}\Bigr)^{r_1}}{{r_1}!} \sum_{s_1=0}^{r_1}(-1)^{r_1+s_1}\binom{r_1}{s_1}\sum_{m_1=0}^{\infty}\binom{\frac{s_1}{2}}{m_1}\Biggl(\frac{24s+1}{24n}\Biggr)^{m_1}\nonumber \\ &=&\sum_{m_1=0}^{\infty}\sum_{r_1=0}^{2m_1}\sum_{s_1=0}^{r_1}\frac{(\frac{\pi}{\sqrt{3}})^{r_1} \Bigl(\frac{24s+1}{24}\Bigr)^{m_1}}{{r_1}!}(-1)^{r_1+s_1} \binom{r_1}{s_1} \binom{\frac{s_1}{2}}{m_1}n^{-\frac{2m_1-r_1}{2}}. \hspace*{0.5cm}
\end{eqnarray}
\endgroup
Since $e^{\frac{\pi}{z\sqrt{3}}\Bigl(\sqrt{1+\frac{24s+1}{24}z^2}-1\Bigr)}$ with $z:=\frac{1}{\sqrt{n}}$ is analytic in the neighbourhood of 0 and thus the Taylor expansion of $e^{\pi\sqrt{\frac{n}{3}}\Bigl(\sqrt{1+\frac{24s+1}{24n}}-1\Bigr)}$ is of the form $\sum_{\ell=0}^{\infty}\frac{a_{\ell}}{\sqrt{n}^{\ell}}$, so the range $0\le r_1 \le \infty$ in the last step (\ref{expeqnmain}) is truncated to $0 \le r_1 \le 2m_1$. \\
Define
$$D':=\{(r_1,s_1,m_1)\in {\mathbb{N}}_0 : 0 \leq s_1 \leq r_1 \}$$
and for $k \in {\mathbb{N}}_0$,
$$D'_k:=\{(r_1,s_1,m_1)\in {\mathbb{N}}_0 : 2m_1-r_1=k \}.$$
For $t_1=(r_1,s_1,m_1)\in D'$, set
$$a'_{t_1}:=\frac{(\frac{\pi}{\sqrt{3}})^{r_1} \left(\frac{24s+1}{24}\right)^{m_1}}{{r_1}!}(-1)^{r_1+s_1} \binom{r_1}{s_1} \binom{\frac{s_1}{2}}{m_1},\ \text{and}\ d_{t_1}:=2m_1-r_1.$$
Therefore \eqref{expeqnmain}) can be rewritten as
\begin{equation}\label{expeqnsum}
e^{\pi\sqrt{\frac{n}{3}}\Bigl(\sqrt{1+\frac{24s+1}{24n}}-1\Bigr)}=\sum_{t_1\in(r_1,s_1,m_1)\in D'}^{}\frac{a'_{t_1}}{n^{\frac{d_{t_1}}{2}}}=\sum_{k\geq 0}\sum_{t_1\in D'_k}\frac{a'_{t_1}}{n^{\frac{k}{2}}}=\sum_{k\geq 0}\sum_{t_1\in D'_{2k}}\frac{a'_{t_1}}{n^{k}}+\sum_{k\geq 0}\sum_{t_1\in D'_{2k+1}}\frac{a'_{t_1}}{n^{k+\frac{1}{2}}}.
\end{equation}
We see that
\begingroup
\allowdisplaybreaks
\begin{align*}
D'_{2k}&=\left\{(r_1,s_1,m_1)\in D' : r_1-2m_1=-2k\right\}\\
&=\left\{(r_1,s_1,m_1)\in D': r_1\equiv 0 \ (\text{mod}\hspace*{0.2cm}2), \ r_1-2m_1=-2k\right\}\\
&=\{(2\ell_1,s_1,\ell_1+k)\in \mathbb{N}_0^3: 0\leq s_1 \leq 2\ell_1\}.
\end{align*}
\endgroup
By Lemma \ref{lemma2.3}, it follows that,
\begingroup
\allowdisplaybreaks
\begin{equation}\label{expeqnsubsum1}
\sum_{k\geq0}^{\infty}\sum_{t\in D'_{2k}}^{}\frac{a'_{t_1}}{n^{k}}=1+\sum_{k \geq 1}^{\infty} \left(\frac{(\frac{24s+1}{24})^k\left(\frac{1}{2}-k\right)_{k+1}}{k} \displaystyle\sum_{\ell_1=1}^{k}\frac{ (-k)_{\ell_1}}{(k+\ell_1)!}\frac{\Bigl(\pi \sqrt{\frac{24s+1}{72}}\Bigr)^{2\ell_1}}{(2\ell_1-1)!}\right)\frac{1}{n^k}= \sum_{k \geq 0}\frac{B_{2k}(s)}{n^k}.	
\end{equation}
\endgroup
Similarly, noting that
$$D'_{2k+1}=\left\{(2\ell_1+1,s_1,\ell_1+k+1) \in {{\mathbb{N}}_0}^3: 0 \leq s_1 \leq 2\ell_1+1\right\}$$
and applying Lemma \ref{lemma2.3} followed by \eqref{expcoeffdef2}, we obtain
\begin{eqnarray}\label{expeqnsubsum2}
&&\hspace{-1 cm}\sum_{k\geq 0}^{\infty}\sum_{t\in D'_{2k+1}}\frac{a'_{t_1}}{n^{k+\frac{1}{2}}}\nonumber\\
\hspace{2 cm}&=&\frac{\pi}{\sqrt{3}}\sum_{k\geq 0}^{\infty} \left(\Bigl(\frac{24s+1}{24}\Bigr)^{k+1}\left(\frac{1}{2}-k\right)_{k+1} \displaystyle\sum_{\ell_1=0}^{k}\frac{ (-k)_{\ell_1}}{(k+\ell_1+1)!}\frac{\Bigl(\pi \sqrt{\frac{24s+1}{72}}\Bigr)^{2\ell_1}}{(2\ell_1)!}\right)\frac{1}{n^{k+\frac 12}} \nonumber \\ &=& \sum_{k \geq 0}\frac{B_{2k+1}(s)}{n^{k+\frac{1}{2}}}.
\end{eqnarray}
Applying \eqref{expeqnsubsum1} and \eqref{expeqnsubsum2} to \eqref{expeqnsum}, we get
\begin{equation*}\label{expeqnseries}
e^{\pi\sqrt{\frac{n}{3}}\Bigl(\sqrt{1+\frac{24s+1}{24n}}-1\Bigr)}= \sum_{k=0}^{N}\frac{B_k(s)}{n^{\frac{k}{2}}}+\sum_{k \geq N+1}^{\infty}\frac{B_k(s)}{n^{\frac{k}{2}}}.
\end{equation*}	
	To estimate $|B_k(s)|$ for $k\in \mathbb{N}$, we estimate $B_{2k}(s)$ and $B_{2k+1}(s)$ separately as follows.
\begingroup
\allowdisplaybreaks
\begin{align}\nonumber
&\left|B_{2k}(s)\right|\le \frac{\left|\Bigl(\frac{24s+1}{24}\Bigr)^k\left(\frac{1}{2}-k\right)_{k+1}\right|}{k}\sum_{\ell_1=1}^{k}\frac{\Bigl(\pi \sqrt{\frac{24s+1}{72}}\Bigr)^{2\ell_1} \left|(-k)_{\ell_1}\right|}{(2\ell_1-1)! (k+\ell_1)!}\\ \nonumber
&\le \frac{\Bigl(\frac{24s+1}{24}\Bigr)^k \binom{2k}{k}}{2k\cdot 4^k}\sum_{\ell_1=1}^{k}\left|\left((-1)^{\ell_1}\prod_{j=0}^{\ell_1-1}\frac{k-j}{k+j-1}\right)\right|\frac{\Bigl(\pi \sqrt{\frac{24s+1}{72}}\Bigr)^{2\ell_1}}{(2\ell_1-1)!} \\\label{expeqncoeff1}
&\le \sqrt{\frac{\pi}{3}}\frac{\Bigl(\frac{24s+1}{24}\Bigr)^{k+\frac 12 } }{2k^{\frac 32}}\sum_{\ell_1=1}^{k}\frac{\Bigl(\pi \sqrt{\frac{24s+1}{72}}\Bigr)^{2\ell_1-1}}{(2\ell_1-1)!}\le \sqrt{\frac{\pi}{3}}\frac{\Bigl(\frac{24s+1}{24}\Bigr)^{k+\frac 12 } }{2k^{\frac 32}}\sinh\left(\pi \sqrt{\frac{24s+1}{72}}\right).
\end{align}
\endgroup
Analogous to \eqref{expeqncoeff1}, for $k\in \mathbb{N}$, we get
\begin{equation}\label{expeqncoeff2}
\hspace{-0.5 cm}\left|B_{2k+1}(s)\right|
\le \sqrt{\frac{\pi}{3}}\frac{\Bigl(\frac{24s+1}{24}\Bigr)^{k+1} }{2k^{\frac 32}}\cosh\Bigl(\pi \sqrt{\frac{24s+1}{72}}\Bigr).
\end{equation}
Combining \eqref{expeqncoeff1} and \eqref{expeqncoeff2}, it follows that for $k\geq 2$,
\begin{equation}\label{Bkfinalbound}
\left|B_k(s)\right| \le \sqrt{\frac{\pi}{3}}\frac{ \Bigl(\frac{24s+1}{24}\Bigr)^{\frac{k+1}{2}} }{2\left(\frac {k-1}2\right)^{\frac{3}{2}}} \cosh\Bigl(\pi \sqrt{\frac{24s+1}{72}}\Bigr),
\end{equation}
and therefore, applying \eqref{Bkfinalbound}, we get for all $N\ge 1$ and $n\ge \lceil \frac{2(24s+1)}{3} \rceil$,
\begingroup
\allowdisplaybreaks
\begin{align}\nonumber
\left|\sum_{k \geq N+1}\frac{B_k(s)}{n^{\frac{k}{2}}}\right|
&
\le \sqrt{\frac{2\pi}{3}}\cosh\Bigl(\pi \sqrt{\frac{24s+1}{72}}\Bigr)\sum_{k \geq N+1}\frac{\left(\frac{24s+1}{24}\right)^{\frac {k+1}2}}{(k-1)^{\frac{3}{2}}n^{\frac k2}}
\\\nonumber
&\le N^{-\frac{3}{2}}\sqrt{\frac{2\pi}{3}}\cdot \sqrt{\frac{24s+1}{24}} \cosh\Bigl(\pi \sqrt{\frac{24s+1}{72}}\Bigr)\sum_{k \geq N+1} \frac{(24s+1)^{\frac{k}{2}}}{ {(24n)}^{\frac{k}{2}}}
\\\nonumber
&\le \frac 43 N^{-\frac{3}{2}}
\sqrt{\frac{2\pi}{3}}\Bigl(\frac{24s+1}{24}\Bigr)^{\frac{N+2}{2}} \cosh\Bigl(\pi \sqrt{\frac{24s+1}{72}}\Bigr)n^{-\frac{N+1}{2}}=Er_2(N,s)n^{-\frac{N+1}{2}}.\nonumber
\end{align}
\endgroup
Thus, we conclude the proof.
\end{proof}

Define
\begin{equation}\label{binomeqncoeffdef}
\overline{A}_\ell\Bigl(\frac{3}{4},s\Bigr):=\begin{cases}
(\frac{24s+1}{24})^{\frac{\ell}{2}}\binom{-3/4}{\frac{\ell}{2}}, &\quad \text{if}\  \ell \equiv 0 \left(\text{mod}\ 2\right),\\
0, &\quad\ \text{otherwise},
\end{cases}
\end{equation}
and
\begin{equation}\label{eqn3}
Er_3\Bigl(\frac{3}{4},N,s\Bigr):=\frac 43 \Bigl(\frac{24s+1}{24}\Bigr)^{\frac{N+1}{2}}.
\end{equation}

\begin{lemma}\label{binom}
	For $N \geq 1$ and $n \geq \left\lceil \frac{2(24s+1)}{3} \right\rceil$, we have
	\begin{equation*}
	\left(1+\frac{24s+1}{24n}\right)^{-\frac{3}{4}}= \sum_{\ell=0}^{N} \frac{\overline{A}_\ell\Bigl(\frac{3}{4},s\Bigr)}{n^{\frac{\ell}{2}}}+ O_{\leq Er_3(\frac{3}{4}, N,s)} \left(n^{-\frac{N+1}{2}}\right).
	\end{equation*}
\end{lemma}
\begin{proof}
	By Taylor expansion of 	$\left(1+\frac{24s+1}{2n}\right)^{-\frac{3}{4}}$, and \eqref{binomeqncoeffdef}, we get
	\begin{equation*}
	\left(1+\frac{24s+1}{24n}\right)^{-\frac{3}{4}}= \sum_{k_1=0}^{\infty} \binom{-\frac{3}{4}}{k_1}  \left(\sqrt{\frac{24s+1}{24n}}\right)^{2k_1}=:\sum_{\ell=0}^{N} \frac{\overline{A}_\ell(\frac{3}{4},s)}{n^{\frac{\ell}{2}}}+\sum_{\ell=N+1}^{\infty} \frac{\overline{A}_\ell(\frac{3}{4},s)}{n^{\frac{\ell}{2}}}.
	\end{equation*}
	
	It is clear that for all $\ell\in \mathbb{N}$,
	\begin{equation}\label{binomeqncoeffbound}
	\left|\overline{A}_\ell\Bigl(\frac{3}{4},s\Bigr)\right|\le \Bigl(\frac{24s+1}{24}\Bigr)^{\frac{\ell}{2}} \ \ \left(\text{since} \ \ \left|\binom{-\frac{3}{4}}{\frac{\ell}{2}}\right| \le 1\right).
	\end{equation}
	Thus for $n \ge \left\lceil \frac{2(24s+1)}{3} \right\rceil$, it follows that
	\begin{equation*}
	\left|\sum_{\ell=N+1}^{\infty}\frac{\overline{A}_\ell(\frac{3}{4},s)}{n^{\frac{\ell}{2}}}\right|\le\sum_{\ell=N+1}^{\infty} \left(\frac{24s+1}{24n}\right)^{\frac{\ell}{2}}\le Er_3\Bigl(\frac{3}{4},N,s\Bigr)\cdot n^{-\frac{N+1}{2}},
	\end{equation*}
	and this finishes the proof.
\end{proof}

Next, we combine Lemmas \ref{exp} and \ref{binom} to estimate an error bound after extracting the main terms for the factor $\frac{e^{\pi\sqrt{\frac{n}{3}}\left(\sqrt{1+\frac{24s+1}{24n}}-1\right)}}
{\Bigl(1+\frac{24s+1}{24n}\Bigr)^{3/4}}$.
Letting $Er_2(N,s)$ and $Er_3\left(\frac{3}{4},N,s\right)$ be as in \eqref{eqn2} and \eqref{eqn3} respectively, we define
\begin{align}\label{Er_{4}(N,s)}
Er_4(N,s)&:=\left(\frac{4}{3}N^{\frac 32}+1\right)Er_2(N,s)+\frac{\pi}{2\sqrt{3}}\Bigl(\frac{24s+1}{24}\Bigr)^{\frac{N+2}{2}}+ Er_3\left(\frac{3}{4}, N,s\right)
\nonumber\\
&\hspace{1 cm}\times\left(1+\frac{\pi}{2\sqrt{3}}\frac{24s+1}{24}
+\frac{\sqrt{\pi(24s+1)}}{72}\cosh\left(\pi \sqrt{\frac{24s+1}{72}}\right)\right).
\end{align}
\begin{lemma}\label{expbinom}
	Let $Er_4(N,s)$ be as given in \eqref{Er_{4}(N,s)}.	For $N \ge 1$ and $n \ge \lceil \frac{2(24s+1)}{3} \rceil$, we have
	\begin{equation*}
	\frac{e^{\pi\sqrt{\frac{n}{3}}\Bigl(\sqrt{1+\frac{24s+1}{24n}}-1\Bigr)}}{\Bigl(1+\frac{24s+1}{24n}\Bigr)^{3/4}}=\sum_{k=0}^{N}\frac{\overline{B}_k(s)}{n^{\frac{k}{2}}}+O_{\le Er_4(N,s)}\left(n^{-\frac{N+1}{2}}\right),
	\end{equation*}
where for $k\in \mathbb{N}_0$,
\begin{equation}\label{expbinomcoeffdef}
\overline{B}_{k}(s):=\sum_{\ell=0}^{k}
B_{\ell}(s)\overline{A}_{k-\ell}\Bigl(\frac{3}{4},s\Bigr).
\end{equation}	
\end{lemma}
\begin{proof}
	Utilizing Lemmas \ref{exp} and \ref{binom}, we get
	\begingroup
	\allowdisplaybreaks
	\begin{align}\nonumber
	\frac{e^{\pi\sqrt{\frac{n}{3}}\Bigl(\sqrt{1+\frac{24s+1}{24n}}-1\Bigr)}}{\Bigl(1+\frac{24s+1}{24n}\Bigr)^{3/4}}&=\left(\sum_{\ell=0}^{N}\frac{B_\ell(s)}{n^{\frac{k}{2}}} + O_{\leq Er_2(N,s)}\left(n^{-\frac{N+1}{2}}\right)\right)\left(\sum_{\ell=0}^{N} \frac{\overline{A}_\ell(\frac 34,s)}{n^{\frac{\ell}{2}}}+ O_{\leq E_3(\frac{3}{4}, N,s)} \left(n^{-\frac{N+1}{2}}\right)\right)\\\nonumber
	&=\sum_{\ell=0}^{N}\frac{B_\ell(s)}{n^{\frac{\ell}{2}}}\sum_{\ell=0}^{N}\frac{\overline{A}_\ell(\frac 34,s)}{n^{\frac{\ell}{2}}}+\sum_{\ell=0}^{N}\frac{B_\ell(s)}{n^{\frac{\ell}{2}}}\cdot O_{\leq Er_3\left(\frac{3}{4}, N,s\right)} \left(n^{-\frac{N+1}{2}}\right) \\ \label{expbinomeqnsum}
	&\hspace{0.5 cm}+\Bigl(1+\frac{24s+1}{24n}\Bigr)^{-3/4}\cdot O_{\leq Er_2(N,s)} \left(n^{-\frac{N+1}{2}}\right). 
	\end{align}	
	\endgroup	
	Clearly, it follows that
\begingroup
\allowdisplaybreaks
\begin{align*}
\sum_{\ell=0}^{N}\frac{B_{\ell}(s)}{n^{\frac{\ell}{2}}}
\sum_{\ell=0}^{N}\frac{\overline{A}_\ell(\frac 34,s)}{n^{\frac{\ell}{2}}}
&=\sum_{k=0}^{N}\frac{1}{n^{\frac{k}{2}}}\sum_{\ell=0}^{k}
B_\ell(s)\overline{A}_{k-\ell}(\frac 34,s)+n^{-\frac{N+1}{2}}\sum_{k=0}^{N-1}\frac{1}{n^{\frac{k}{2}}}
\sum_{\ell=k}^{N-1}B_{\ell+1}(s)\overline{A}_{N+k-\ell}(\frac 34,s)
\\
&=:\sum_{k=0}^{N}\frac{\overline{B}_k(s)}{n^{\frac k2}}+T_1(N,s,n).
\end{align*}	
\endgroup	
Applying \eqref{Bkfinalbound} and \eqref{binomeqncoeffbound}, for $n\ge \lceil \frac{2(24s+1)}{3} \rceil$, we have	
\begin{align*}
&\left|T_1(N,s,n)\right|
\le n^{-\frac{N+1}{2}}\Biggl(\sum_{k=1}^{N-1}\frac{1}{n^{\frac{k}{2}}}
\sum_{\ell=k}^{N-1}\sqrt{\frac{2\pi}{3}}\frac{ \Bigl(\frac{24s+1}{24}\Bigr)^{\frac{N+k+2}{2}} }{\ell^{\frac{3}{2}}} \cosh\Bigl(\pi \sqrt{\frac{24s+1}{72}}\Bigr)
\\
&\hspace{3.8cm}+\frac{\pi}{2\sqrt{3}}\Bigl(\frac{24s+1}{24}\Bigr)^{\frac{N+2}{2}}
+\sum_{\ell=1}^{N-1}\sqrt{\frac{2\pi}{3}}\frac{ \Bigl(\frac{24s+1}{24}\Bigr)^{\frac{N+2}{2}} }{\ell^{\frac{3}{2}}} \cosh\Bigl(\pi \sqrt{\frac{24s+1}{72}}\Bigr)\Biggr)
\\
&\le n^{-\frac{N+1}{2}}\Biggl(\frac{N^{\frac 32}}{2}Er_2(N,s)\left(\sum_{k=1}^{N-1}
\Bigl(\frac{24s+1}{24n}\Bigr)^{\frac{k}{2}}
\sum_{\ell=k}^{N-1}\ell^{-\frac{3}{2}}
+\sum_{\ell=1}^{N-1}\ell^{-\frac{3}{2}}\right)
+\frac{\pi}{2\sqrt{3}}\Bigl(\frac{24s+1}{24}\Bigr)^{\frac{N+2}{2}}\Biggr)
\\
&\le n^{-\frac{N+1}{2}}\Biggl(\frac{4}{3}N^{\frac 32}Er_2(N,s)
+\frac{\pi}{2\sqrt{3}}\Bigl(\frac{24s+1}{24}\Bigr)^{\frac{N+2}{2}}\Biggr),
\end{align*}
where in the last step, we use that
\begin{equation}\label{bound-int}
\sum_{\ell=1}^{N-1}\ell^{-\frac{3}{2}}
\leq \sum_{\ell=1}^{\infty}\ell^{-\frac{3}{2}}=\zeta\left(\frac 32\right)\leq 2.
\end{equation}
Then letting
\begin{equation*}
Er_{4,1}(N,s):=\frac{4}{3}N^{\frac 32}Er_2(N,s)
+\frac{\pi}{2\sqrt{3}}\Bigl(\frac{24s+1}{24}\Bigr)^{\frac{N+2}{2}},
\end{equation*}
for all $N\ge 2$ and $n\ge \lceil \frac{2(24s+1)}{3} \rceil$, we obtain
\begin{equation}\label{expbinomMTsumfinal}
\sum_{\ell=0}^{N}\frac{B_k(s)}{n^{\frac{\ell}{2}}}
\sum_{\ell=0}^{N}\frac{\overline{A}_\ell(\frac 34,s)}{n^{\frac{\ell}{2}}}
=\sum_{k=0}^{N}\frac{\overline{B}_k(s)}{n^{\frac k2}}+O_{\le Er_{4,1}(N,s)}\left(n^{-\frac{N+1}{2}}\right).
\end{equation}	

	To finish the proof, it remains to estimate the two terms involving the factor $n^{-\frac{N+1}{2}}$ in \eqref{expbinomeqnsum}. It is clear that for all $n\ge 1$,
	\begin{equation}\label{expbinomsum4}
	\Bigl(1+\frac{24s+1}{24n}\Bigr)^{-3/4}\cdot O_{\leq Er_2(N,s)} \left(n^{-\frac{N+1}{2}}\right)=O_{\le Er_2(N,s)}\left(n^{-\frac{N+1}{2}}\right).
	\end{equation}
	For all $n\ge \lceil \frac{2(24s+1)}{3} \rceil$, employing \eqref{Bkfinalbound} gives that
	\begin{align*}
	\sum_{\ell=0}^{N}\frac{B_\ell(s)}{n^{\frac{\ell}{2}}}
	&\leq 1+\frac{\pi}{2\sqrt{3}}\frac{24s+1}{24}
	+\sum_{\ell=2}^{N}\sqrt{\frac{2\pi}{3}}\frac{ \Bigl(\frac{24s+1}{24}\Bigr)^{\frac{\ell+1}{2}} }{n^{\frac{\ell}{2}}(\ell-1)^{\frac{3}{2}}} \cosh\Bigl(\pi \sqrt{\frac{24s+1}{72}}\Bigr)
	\\ \notag
	&\leq 1+\frac{\pi}{2\sqrt{3}}\frac{24s+1}{24}
	+\frac{\sqrt{\pi(24s+1)}}{72}\cosh\Bigl(\pi \sqrt{\frac{24s+1}{72}}\Bigr).
	\end{align*}
	and therefore,
	\begin{align}\label{expbinomsum2}
	\sum_{k=0}^{N}\frac{B_k(s)}{n^{\frac{k}{2}}}\cdot O_{\leq Er_3\left(\frac{3}{4},N,s\right)} \left(n^{-\frac{N+1}{2}}\right)&\nonumber\\
	&\hspace{-3 cm}=O_{\le \left(1+\frac{\pi}{2\sqrt{3}}\frac{24s+1}{24}
		+\frac{\sqrt{\pi(24s+1)}}{72}\cosh\left(\pi \sqrt{\frac{24s+1}{72}}\right)\right) Er_3\left(\frac{3}{4}, N,s\right)}\left(n^{-\frac{N+1}{2}}\right).
	\end{align}
	Finally, combining \eqref{expbinomMTsumfinal}-\eqref{expbinomsum2} to \eqref{expbinomeqnsum}, we conclude the proof.
\end{proof}

For all $\ell\in \mathbb{N}_0$, define
\begin{align}\nonumber	\widehat{C}_{2\ell}(s)&:=\sum_{k=0}^{\ell}\binom{-k}{\ell-k}\frac{a_{2k}(1)}{(\frac{\pi}{\sqrt{3}})^{2k}}\Bigl(\frac{24s+1}{24}\Bigr)^{\ell-k},\\\label{besselcoeffdef} \widehat{C}_{2\ell+1}(s)&:=-\sum_{k=0}^{\ell}\binom{-\frac{2k+1}{2}}{\ell-k}\frac{a_{2k+1}(1)}{(\frac{\pi}{\sqrt{3}})^{2k+1}}\Bigl(\frac{24s+1}{24}\Bigr)^{\ell-k}.
\end{align}
\begin{lemma}\label{bessel}
	For $N \geq1$ and $n\ge \max\left\{\Bigl \lceil \frac{2(24s+1)}{3} \Bigr \rceil,\frac{48}{\pi^2}  \right\}$, we have
	\begin{equation*}
	\sum_{m=0}^{N}\frac{(-1)^m a_m(1)}{(\frac{\pi}{\sqrt{3}})^m n^{\frac{m}{2}}}\Biggl(1+\frac{24s+1}{24n}\Biggr)^{-\frac{m}{2}} =\sum_{m=0}^{N} \frac{\widehat{C}_m(s)}{n^{\frac{m}{2}}}+O_{\leq Er_5(N,s)}\left(n^{-\frac{N+1}{2}}\right),
	\end{equation*}
	where
	\begin{equation}\label{eqn5}
	\hspace{-0.2 cm}Er_5(N,s):=\frac{4\cdot |a_N(1)|}{3}
	\Bigl(\frac{24s+1}{24}\!+\!\frac{3}{\pi^2}\Bigr)^{\left\lfloor \frac{N}{2} \right\rfloor+1}\!+\!\frac{4\cdot|a_N(1)|}{\sqrt{3}\pi } \!\Biggl(\sqrt{\frac{24s+1}{24}}+\frac{\sqrt{3} }{\pi }\Biggr)^{2\left\lfloor \frac{N-1}{2} \right\rfloor+2}.
	\end{equation}
\end{lemma}
\begin{proof}
	We set
$$\widehat{b}_k:=\frac{(-1)^k a_k(1)}{(\frac{\pi}{\sqrt{3}})^k},\  \widehat{c}_{k,\ell}(s):=\Bigl(\frac{24s+1}{24}\Bigr)^\ell\binom{-\frac{k}{2}}{\ell},\ \text{and}\ w:=n^{-\frac 12},$$
and consequently, rewrite the sum as
\begingroup
\allowdisplaybreaks
\begin{align}\nonumber
\sum_{m=0}^{N}\frac{(-1)^m a_m(1)}{(\frac{\pi}{\sqrt{3}})^m n^{\frac{m}{2}}}\Biggl(1+\frac{24s+1}{24n}\Biggr)^{-\frac{m}{2}}
&=\sum_{k=0}^{N} \sum_{\ell=0}^{\infty} \widehat{b}_k \widehat{c}_{k,\ell}(s) w^{2\ell+k}\\\nonumber
&=\sum_{m=0}^{N}\widehat{C}_m(s)w^m+\sum_{m=\left\lfloor \frac{N}{2} \right\rfloor+1}^{\infty}\sum_{k=0}^{\left\lfloor \frac{N}{2} \right\rfloor} \widehat{b}_{2k} \widehat{c}_{2k,m-k}(s) w^{2m}\\\nonumber
&\hspace{1 cm}+\sum_{m=\left\lfloor \frac{N-1}{2} \right\rfloor+1}^{\infty} \sum_{k=0}^{\left\lfloor \frac{N-1}{2} \right\rfloor} \widehat{b}_{2k+1} \widehat{c}_{2k+1,m-k}(s) w^{2m+1}\\\label{besseleqn}
&=:\sum_{m=0}^{N} \widehat{C}_m(s) w^m +\widehat{E}_1\left(N,w\right)+ \widehat{E}_2\left(N,w\right).
\end{align}
\endgroup
On the one hand, we have
\begingroup
\allowdisplaybreaks
\begin{align}\nonumber
&\left|\widehat{E}_1(N,w)\right|\le \sum_{m=\left\lfloor \frac{N}{2} \right\rfloor+1}^{\infty} \sum_{k=0}^{\left\lfloor \frac{N}{2} \right\rfloor} |\widehat{b}_{2k}| |\widehat{c}_{2k,m-k}(s)| w^{2m}\ \ \left(\text{by}\ \eqref{besseleqn}\right)\\\nonumber
&=\sum_{m=\left\lfloor \frac{N}{2} \right\rfloor+1}^{\infty}\frac{1}{n^m} \sum_{k=0}^{\left\lfloor \frac{N}{2} \right\rfloor}\left|\frac{a_{2k}(1)}{(\frac{\pi}{\sqrt{3}})^{2k}}\right|\left|\binom{-k}{m-k}\right|\Bigl(\frac{24s+1}{24}\Bigr)^{m-k}\\\nonumber
&=\sum_{m=\left\lfloor \frac{N}{2} \right\rfloor+1}^{\infty}\frac{1}{n^m} \sum_{k=0}^{\left\lfloor \frac{N}{2} \right\rfloor}\binom{m-1}{k-1}\left|\frac{a_{2k}(1)}{(\frac{\pi}{\sqrt{3}})^{2k}}\right|\Bigl(\frac{24s+1}{24}\Bigr)^{m-k}\\\nonumber
&\le |a_N(1)|\sum_{m=\left\lfloor \frac{N}{2} \right\rfloor+1}^{\infty}\Bigl(\frac{24s+1}{24n}\Bigr)^m \sum_{k=0}^{m-1}\binom{m-1}{k-1}\left( \frac{3}{\pi^2}\cdot \frac{24}{24s+1}\right)^{k}\nonumber \\
&\le |a_N(1)|\sum_{m=\left\lfloor \frac{N}{2} \right\rfloor+1}^{\infty}\Bigl(\frac{24s+1}{24n}+\frac{3}{\pi^2 n}\Bigr)^m
\nonumber\\
&\le \frac{4\cdot |a_N(1)|}{3}
\Bigl(\frac{24s+1}{24n}+\frac{3}{\pi^2 n}\Bigr)^{\left\lfloor \frac{N}{2} \right\rfloor+1} \ \ \left(\forall\ n\ge \max\left\{\Bigl \lceil \frac{24s+1}{3} \Bigr \rceil,\frac{24}{\pi^2}  \right\}\right) \nonumber \\ \label{besseleqnER1}
&\le \frac{4\cdot |a_N(1)|}{3}
\Bigl(\frac{24s+1}{24}+\frac{3}{\pi^2}\Bigr)^{\left\lfloor \frac{N}{2} \right\rfloor+1} n^{-\frac{N+1}{2}}.	\end{align}
\endgroup
On the other hand, we estimate
\begingroup
\allowdisplaybreaks
\begin{align}\nonumber
&\left|\widehat{E}_2(N,w)\right|\\\nonumber
&\le \sum_{m=\left\lfloor \frac{N-1}{2} \right\rfloor+1}^{\infty} \frac{1}{n^{m+\frac{1}{2}}} \sum_{k=0}^{\left\lfloor \frac{N-1}{2} \right\rfloor} |\widehat{b}_{2k+1}| |\widehat{c}_{2k+1,m-k}(s)|\ \ \left(\text{by}\ \eqref{besseleqn}\right)\\\nonumber
&=\sum_{m=\left\lfloor \frac{N-1}{2} \right\rfloor+1}^{\infty}\frac{1}{n^{m+\frac{1}{2}}} \sum_{k=0}^{\left\lfloor \frac{N-1}{2} \right\rfloor}\left|\frac{a_{2k+1}(1)}{(\frac{\pi}{\sqrt{3}})^{2k+1}}\right|\left|\frac{\binom{2m}{2k}\binom{2m-2k}{m-k}}{4^{m-k}\binom{m}{k}}\right|\Bigl(\frac{24s+1}{24}\Bigr)^{m-k}\\\nonumber
&\le \frac{\sqrt{3} \cdot |a_N(1)|}{\pi} \sum_{m=\left\lfloor \frac{N-1}{2} \right\rfloor+1}^{\infty}\frac{1}{n^{m+\frac{1}{2}}} \sum_{k=0}^{\left\lfloor \frac{N-1}{2} \right\rfloor}\binom{2m}{2k}\left( \frac{3}{\pi^2}\cdot \frac{24}{24s+1}\right)^k \Biggl(\frac{24s+1}{24}\Biggr)^m
\nonumber\\
&\le \frac{\sqrt{3} \cdot |a_N(1)|}{\pi n^{\frac{1}{2}}} \sum_{m=\left\lfloor \frac{N-1}{2} \right\rfloor+1}^{\infty}\Biggl(\frac{24s+1}{24n}\Biggr)^{m}  \sum_{k=0}^{2m}\binom{2m}{2k} \left( \frac{3}{\pi^2}\cdot \frac{24}{24s+1}\right)^k
\nonumber \\
&\le \frac{\sqrt{3} \cdot |a_N(1)|}{\pi n^{\frac{1}{2}}} \sum_{m=\left\lfloor \frac{N-1}{2} \right\rfloor+1}^{\infty}\Biggl(\sqrt{\frac{24s+1}{24n}}+\frac{\sqrt{3} }{\pi \sqrt{n}}\Biggr)^{2m}
\nonumber\\ \label{besseleqnER2}
&\le  \frac{4\cdot|a_N(1)|}{\sqrt{3}\pi } \Biggl(\sqrt{\frac{24s+1}{24}}+\frac{\sqrt{3} }{\pi }\Biggr)^{2\left\lfloor \frac{N-1}{2} \right\rfloor+2} n^{-\frac{N+1}{2}} \ \ \left(\forall\ n\ge \max\left\{\Bigl \lceil \frac{2(24s+1)}{3} \Bigr \rceil,\frac{48}{\pi^2}  \right\}\right).
\end{align}
\endgroup
Applying \eqref{besseleqnER1} and \eqref{besseleqnER2} to \eqref{besseleqn}, we conclude the proof.	
\end{proof}

	\begin{lemma}\label{besselER}
	Let $Er_{N,1}$ and $Er_5(N,s)$ be as given in \eqref{eqn1} and \eqref{eqn5} separately. For $N \ge 1$ and $n\ge \max\left\{\Bigl \lceil \frac{2(24s+1)}{3} \Bigr \rceil,\frac{48}{\pi^2}  \right\}$, we have
	\begingroup
	\allowdisplaybreaks
	\begin{align*}
	&\left(\sum_{m=0}^{N}\frac{(-1)^m a_m(1)}{(\frac{\pi}{\sqrt{3}})^m n^{\frac{m}{2}}}\Biggl(1+\frac{24s+1}{24n}\Biggr)^{-\frac{m}{2}}+O_{\le Er_{N,1}}\left(n^{-\frac{N+1}{2}}\right) \right) \left(1+O_{\le 4}\left(\frac{3^{\frac{N+1}{2}}}{\pi^{N+1}}n^{-\frac{N+1}{2}}\right)\right)\\
	&\hspace{7.5 cm}=\sum_{m=0}^{N} \frac{\widehat{C}_m(s)}{n^{\frac{m}{2}}}+O_{\leq Er_6(N,s)}\left(n^{-\frac{N+1}{2}}\right),
	\end{align*}
	\endgroup
	where
	\begin{align}\label{eqn6}
	Er_6(N,s):=8 \left(\frac{\sqrt{3}}{\pi}\right)^{N+1} |a_N(1)|
	+\left(1+4\left(\frac{3}{\pi\sqrt{2(24s+1)}}\right)^{N+1}\right)(Er_5(N,s)+Er_{N,1}).
	\end{align}
\end{lemma}
\begin{proof}
	Using Lemma \ref{bessel}, for all $n\ge \max\left\{\Bigl \lceil \frac{2(24s+1)}{3} \Bigr \rceil,\frac{48}{\pi^2}  \right\}$, we get
	\begingroup
	\allowdisplaybreaks
	\begin{align*}
	\sum_{m=0}^{N}\frac{(-1)^m a_m(1)}{(\frac{\pi}{\sqrt{3}})^m n^{\frac{m}{2}}}\Biggl(1+\frac{24s+1}{24n}\Biggr)^{-\frac{m}{2}}&+O_{\le Er_{N,1}}\left(n^{-\frac{N+1}{2}}\right)\\
	&=\sum_{m=0}^{N} \frac{\widehat{C}_m(s)}{n^{\frac{m}{2}}} +O_{\leq {(Er_5(N,s)+Er_{N,1})}}\left(n^{-\frac{N+1}{2}}\right),
	\end{align*}
	\endgroup
	and hence,
	\begin{align}\nonumber
	&\left(\sum_{m=0}^{N}\frac{(-1)^m a_m(1)}{(\frac{\pi}{\sqrt{3}})^m n^{\frac{m}{2}}}\Biggl(1+\frac{24s+1}{24n}\Biggr)^{-\frac{m}{2}}+O_{\le Er_{N,1}}\left(n^{-\frac{N+1}{2}}\right) \right) \left(1+O_{\le 4}\left(\frac{3^{\frac{N+1}{2}}}{\pi^{N+1}}n^{-\frac{N+1}{2}}\right)\right)\\\label{besselERsum}
	&=\sum_{m=0}^{N} \frac{\widehat{C}_m(s)}{n^{\frac{m}{2}}}+\sum_{m=0}^{N} \frac{\widehat{C}_m(s)}{n^{\frac{m}{2}}}\cdot O_{\le 4}\left(\frac{3^{\frac{N+1}{2}}}{\pi^{N+1}}n^{-\frac{N+1}{2}}\right)
	+O_{\leq Er_{6,1}(N,s)}\left(n^{-\frac{N+1}{2}}\right),
	\end{align}
	where
	\begin{equation*}
	Er_{6,1}(N,s):=  \left(1+4\left(\frac{3}{\pi\sqrt{2(24s+1)}}\right)^{N+1}\right)
	(Er_5(N,s)+Er_{N,1}).
	\end{equation*}
	For all $n\ge \max\left\{\Bigl \lceil \frac{2(24s+1)}{3} \Bigr \rceil,\frac{48}{\pi^2}  \right\}$, we obtain
	\begingroup
	\allowdisplaybreaks
	\begin{align*}
	\left|\sum_{m=0}^{N} \frac{\widehat{C}_m(s)}{n^{\frac{m}{2}}}\right|&\le \sum_{m=0}^{\left\lfloor \frac{N}{2} \right\rfloor} \frac{\left |\widehat{C}_{2m}(s)\right|} {n^{m}}+\sum_{m=0}^{\left\lfloor \frac{N-1}{2} \right\rfloor} \frac{\left|\widehat{C}_{2m+1}(s)\right|} {n^{m+\frac{1}{2}}}\\
	&\le \sum_{m=0}^{\left\lfloor \frac{N}{2} \right\rfloor}\frac{1}{n^m}\sum_{k=0}^{m}\binom{m-1}{k-1}\frac{\left|a_{2k}(1)\right|}{(\frac{\pi}{\sqrt{3}})^{2k}}\Bigl(\frac{24s+1}{24}\Bigr)^{m-k}\\
	&+\sum_{m=0}^{\left\lfloor \frac{N-1}{2} \right\rfloor} \frac{1}{n^{m+\frac 12}}\sum_{k=0}^{m}\left|\binom{-\frac{2k+1}{2}}{m-k}\right|\frac{\left|a_{2k+1}(1)\right|}{(\frac{\pi}{\sqrt{3}})^{2k+1}}\Bigl(\frac{24s+1}{24}\Bigr)^{m-k}\ \left(\text{by Lemma}\ \ref{bessel}\right)\\
	&=\sum_{m=0}^{\left\lfloor \frac{N}{2} \right\rfloor}\frac{1}{n^m}\sum_{k=0}^{m}\binom{m-1}{k-1}\frac{\left|a_{2k}(1)\right|}{(\frac{\pi}{\sqrt{3}})^{2k}}\Bigl(\frac{24s+1}{24}\Bigr)^{m-k}\\
	&\hspace{3 cm}+\sum_{m=0}^{\left\lfloor \frac{N-1}{2} \right\rfloor} \frac{1}{n^{m+\frac 12}}\sum_{k=0}^{m}\frac{\binom{2m}{2k}\binom{2m-2k}{m-k}}{4^{m-k}\binom{m}{k}}\frac{\left|a_{2k+1}(1)\right|}{(\frac{\pi}{\sqrt{3}})^{2k+1}}\Bigl(\frac{24s+1}{24}\Bigr)^{m-k}\\
	&\le 2 \cdot |a_N(1)|.
	\end{align*}
	\endgroup
	Therefore,
	\begin{equation}\label{besselERsum2}
	\sum_{m=0}^{N} \frac{\widehat{C}_m(s)}{n^{\frac{m}{2}}}\cdot O_{\le 4}\left(\frac{3^{\frac{N+1}{2}}}{\pi^{N+1}}n^{-\frac{N+1}{2}}\right)=O_{\le 8\cdot 3^{\frac{N+1}{2}}\cdot\pi^{-N-1}|a_N(1)|}\left(n^{-\frac{N+1}{2}}\right).
	\end{equation}
	Applying \eqref{besselERsum2} to \eqref{besselERsum}, we conclude the proof.
\end{proof}

	\section{Proof of Theorem 1.1}\label{sec-asym-thm}
In this section, by making use of the estimations of each
terms in the asymptotic expansion which were presented in Section \ref{sec-pre}, we prove Theorem \ref{MainTh}.

{\noindent\it Proof of Theorem \ref{MainTh}.}	
Applying Lemmas \ref{expbinom} and \ref{besselER} to
\eqref{q(n)second}, for $n \ge n(N,s)$ (cf. \eqref{n(N,s)}), we have
\begingroup
\allowdisplaybreaks
\begin{align*}
&q(n+s)\\
&=\frac{e^{ \pi \sqrt{\frac{n}{3}}}}{4\cdot3^{1/4} n^{\frac{3}{4}}}\left(\sum_{m=0}^{N}\frac{\overline{B}_m(s)}{n^{\frac{m}{2}}}+O_{\le Er_4(N,s)}\left(n^{-\frac{N+1}{2}}\right)\right)\left(\sum_{m=0}^{N} \frac{\widehat{C}_m(s)}{n^{\frac{m}{2}}}+O_{\leq Er_6(N,s)}\left(n^{-\frac{N+1}{2}}\right)\right).
\end{align*}
\endgroup
Now, expanding the above, it gives
\begingroup
\allowdisplaybreaks
\begin{align*}
&\left(\sum_{m=0}^{N}\frac{\overline{B}_m(s)}{n^{\frac{m}{2}}}+O_{\le Er_4(N,s)}\left(n^{-\frac{N+1}{2}}\right)\right)\left(\sum_{m=0}^{N} \frac{\widehat{C}_m(s)}{n^{\frac{m}{2}}}+O_{\leq Er_6(N,s)}\left(n^{-\frac{N+1}{2}}\right)\right)
\\
&=\sum_{m=0}^{N}\sum_{k=0}^{m}\frac{\overline{B}_k(s)\cdot \widehat{C}_{m-k}(s)}{n^{\frac m2}}+n^{-\frac{N+1}{2}}\sum_{m=0}^{N-1}\frac{1}{n^{\frac m2}}\sum_{k=m}^{N-1}\overline{B}_{k+1}(s)\cdot \widehat{C}_{N+m-k}(s)
\\
&\hspace{3 cm}+\sum_{m=0}^{N}\frac{\overline{B}_m(s)}{n^{\frac{m}{2}}}\cdot O_{\leq Er_6(N,s)}\left(n^{-\frac{N+1}{2}}\right)+\sum_{m=0}^{N} \frac{\widehat{C}_m(s)}{n^{\frac{m}{2}}}\cdot O_{\le Er_4(N,s)}\left(n^{-\frac{N+1}{2}}\right)
\\
&\hspace{8.75 cm}+O_{\le {\frac{Er_4(N,s)\cdot Er_6(N,s)}{n(N,s)^{\frac{N+1}{2}}}}}\left(n^{-\frac{N+1}{2}}\right)
\\
&=:\sum_{m=0}^{N}\frac{\widehat{B}_m(s)}{n^{\frac m2}}+\overline{Er}_1(n, N,s)+\overline{Er}_2(n, N,s)+\overline{Er}_3(n, N,s)+O_{\le {\frac{Er_4(N,s)\cdot Er_6(N,s)}{n(N,s)^{\frac{N+1}{2}}}}}\left(n^{-\frac{N+1}{2}}\right),
\end{align*}
\endgroup
where for $\ell\in \mathbb{N}_0$,
\begin{equation}\label{finalcoeffdef}
\widehat{B}_\ell(s):=\sum_{k=0}^{\ell}\overline{B}_k(s)\cdot \widehat{C}_{\ell-k}(s).
\end{equation}
	We first estimate $\overline{B}_m(s)$ and $\widehat{C}_m(s)$. For $m\geq 2$, Using \eqref{Bkfinalbound} and \eqref{bound-int}, we have
\begin{align}\label{BbarfinalBD}
\left|\overline{B}_{m}(s)\right|
&=\left|\sum_{\ell=0}^{m}B_{\ell}(s)\overline{A}_{m-\ell}\Bigl(\frac{3}{4},s\Bigr)\right|\ \left(\text{by}\ \eqref{expbinomcoeffdef}\right)
\nonumber\\
&\le B_0(s)\left|\overline{A}_{m}\Bigl(\frac{3}{4},s\Bigr)\right|
+\left|B_{1}(s)\right|\overline{A}_{m-1}\Bigl(\frac 34,s\Bigr)
+\sum_{\ell=2}^{m}\left|B_{\ell}(s)\right|
\cdot\left|\overline{A}_{m-\ell}\Bigl(\frac{3}{4},s\Bigr)\right|
\nonumber\\
&\le \left(\frac{24s+1}{24}\right)^{\frac m2}\Biggl(1+\frac{\pi}{2\sqrt{3}}\sqrt{\frac{24s+1}{24}}
+\sqrt{\frac{2\pi}{3}}\sqrt{\frac{24s+1}{24}}
\cosh\Bigl(\pi \sqrt{\frac{24s+1}{72}}\Bigr)\sum_{\ell=2}^{m}(l-1)^{-\frac{3}{2}}\Biggr)
\nonumber\\
&\le \left(\frac{24s+1}{24}\right)^{\frac m2}\Biggl(1+\frac{\pi}{2\sqrt{3}}\sqrt{\frac{24s+1}{24}}+
\frac{\sqrt{\pi(24s+1)}}{12}\cosh\Bigl(\pi \sqrt{\frac{24s+1}{72}}\Bigr)\Biggr).
\end{align}
Next, we have for all $m\in \mathbb{N}$ as
\begingroup
\allowdisplaybreaks
\begin{align}\nonumber
\left|\widehat{C}_{2m}(s)\right|&\le \sum_{k=0}^{m}\left|\binom{-k}{m-k}\right|\frac{\left|a_{2k}(1)\right|}{(\frac{\pi}{\sqrt{3}})^{2k}}\left(\frac{24s+1}{24}\right)^{m-k}\\\label{ChatBD1}
&\le |a_{2m}(1)|\left(\frac{24s+1}{24}\right)^{m}\sum_{k=0}^{m}\binom{m-1}{k-1}\le |a_{2m}(1)|\cdot \left(\frac{24s+1}{12}\right)^{m},
\end{align}
\endgroup
and in a similar fashion, we get for $m\in \mathbb{N}$,
\begin{equation}\label{ChatBD2}
\left|\widehat{C}_{2m+1}(s)\right|\le  |a_{2m+1}(1)|\cdot  \left(\frac{24s+1}{6}\right)^{m}.
\end{equation}
From \eqref{ChatBD1}, \eqref{ChatBD2} and the fact that $\widehat{C}_0(s)=1$ (by \eqref{besselcoeffdef}, it follows that
\begin{equation}\label{ChatfinalBD}
\left|\widehat{C}_m(s)\right|\le  |a_{m}(1)|\cdot \left(\frac{24s+1}{6}\right)^{\frac m2}.
\end{equation}
With the above estimation, now we can estimate $\overline{Er}_{i}(n,N,s)$ for all $i=1,2,3$.
Setting
\begin{equation}\label{eqn8}
A(s):=1+\frac{\pi}{2\sqrt{3}}\sqrt{\frac{24s+1}{24}}+
\frac{\sqrt{\pi(24s+1)}}{12}\cosh\Bigl(\pi \sqrt{\frac{24s+1}{72}}\Bigr),
\end{equation}
by \eqref{BbarfinalBD} and \eqref{ChatfinalBD}, we see that, for $n \ge \left\lceil \frac{2(24s+1)}{3} \right \rceil$,
\begingroup
\allowdisplaybreaks
\begin{align}\nonumber
&\left|\overline{Er}_1(n,N,s)\right|\le \left| n^{-\frac{N+1}{2}}\sum_{m=0}^{N-1}\frac{1}{n^{\frac m2}}\sum_{k=m}^{N-1}\overline{B}_{k+1}(s)\cdot \widehat{C}_{N+m-k}(s)\right|\\\nonumber
&\le \Biggl(\frac{\pi|a_N(1)|\cdot2^{N-1}}{\sqrt{3}}\left(\frac{24s+1}{24}\right)
^{\frac{N}{2}+1}
+A(s)\sum_{k=1}^{N-1}\left(\frac{24s+1}{24}\right)^{\frac{k+1}{2}}
|a_{N-k}(1)|\left(\frac{24s+1}{6}\right)^{\frac{N-k}{2}}
\\\nonumber
&\hspace{2cm}+A(s)\sum_{m=1}^{N-1}\frac{1}{n^{\frac m2}}\sum_{k=m}^{N-1}\left(\frac{24s+1}{24}\right)^{\frac{k+1}{2}}
|a_{N+m-k}(1)|\left(\frac{24s+1}{6}\right)^{\frac{N+m-k}{2}}\Biggr)
n^{-\frac{N+1}{2}}\\\nonumber
&\le \Biggl(\frac{\pi|a_N(1)|\cdot2^{N-1}}{\sqrt{3}}\left(\frac{24s+1}{24}\right)
^{\frac{N}{2}+1}+A(s)|a_N(1)| \Bigl(\frac{24s+1}{24}\Bigr)^{\frac{N+1}{2}}
\\\nonumber
&\hspace{2cm}+A(s)|a_N(1)|\cdot \Bigl(\frac{24s+1}{6}\Bigr)^{\frac{N}{2}}\left(\frac{24s+1}{24}\right)
^{\frac 12}\sum_{m= 1}^{N-1}\frac{1}{2^m}
\sum_{k=m}^{N-1}\frac{1}{2^k}\Biggr)n^{-\frac{N+1}{2}}\\\label{finaleqnsumER1}
&\le |a_N(1)|\left(\frac{\pi\cdot2^{N-1}}{\sqrt{3}}\sqrt{\frac{24s+1}{24}}+A(s) \left(1+\frac{2^{N+1}}{3}\right) \right)\Bigl(\frac{24s+1}{24}\Bigr)^{\frac{N+1}{2}}n^{-\frac{N+1}{2}}.
\end{align}
\endgroup
For all $n \ge \left\lceil \frac{2(24s+1)}{3} \right \rceil$, we get
\begingroup
\allowdisplaybreaks
\begin{align}\label{eqn9}
\left|\overline{Er}_2(n,N,s)\right|
&\le Er_6(N,s)\cdot n^{-\frac{N+1}{2}}\sum_{m=0}^{N}\frac{\left|\overline{B}_m(s)\right|}{n^{\frac m2}}
\nonumber\\
&\le \left(1+\frac{\pi}{2\sqrt{3}}\frac{24s+1}{24}+ \frac{A(s)}{12} \right) Er_6(N,s) n^{-\frac{N+1}{2}},
\end{align}
\endgroup
and
\begingroup
\allowdisplaybreaks
\begin{align}\label{eqn10}
\left|\overline{Er}_3(n,N,s)\right|&\le Er_4(N,s)\cdot n^{-\frac{N+1}{2}}\sum_{m=0}^{N}\frac{\left|\widehat{C}_m(s)\right|}{n^{\frac m2}}\le 2|a_N(1)|Er_4(N,s) n^{-\frac{N+1}{2}}.
\end{align}
\endgroup
Combining \eqref{finaleqnsumER1}, \eqref{eqn9}, and \eqref{eqn10} with following \eqref{finalerrordef}, we finish the proof of Theorem \ref{MainTh}.
\qed

\section{Inequalities for $q(n)$}\label{sec-ineq}
In this section, we prove all the inequalities by applying Theorem \ref{MainTh} and choosing appropriate values of $N$.

From Theorem \ref{MainTh}, we have for $(s,N)\in \mathbb{N}_0\times \mathbb{N}$ and $n\ge n(N,s)$,
\begin{equation}\label{Maintheqn1}
\frac{e^{\pi\sqrt{\frac{n}{3}}}}{4\cdot3^{1/4}n^{3/4}}L\left(n,s, N\right)\le q(n+s)\le \frac{e^{\pi\sqrt{\frac{n}{3}}}}{4\cdot3^{1/4}n^{3/4}}U\left(n,s, N\right),
\end{equation}
where
\begin{align*}
L\left(n,s, N\right):=\sum_{m=0}^{N}\frac{\widehat{B}_m(s)}{n^{\frac m2}}-\frac{Er_N(s)}{n^{\frac{N+1}{2}}},\quad U\left(n,s, N\right):=\sum_{m=0}^{N}\frac{\widehat{B}_m(s)}{n^{\frac m2}}+\frac{Er_N(s)}{n^{\frac{N+1}{2}}}.
\end{align*}

{\noindent\it Proof of Theorem \ref{A-thm}.}
Making the shift $n\to n+1$, it is equivalent to show that for $n\ge 229$,
$$A\left(q(n),q(n+1), q(n+2), q(n+3), q(n+4)\right)>0.$$
Following \eqref{Maintheqn1}, we have
\begin{align*}
A\left(q(n),q(n+1), q(n+2), q(n+3), q(n+4)\right)&\\
&\hspace{-7 cm}>\left(\frac{e^{\pi\sqrt{\frac{n}{3}}}}{4\cdot3^{1/4}n^{3/4}}\right)^2\left(L(n,0,N)L(n,4,N)+3L^2(n,2,N)-4U(n,1,N)U(n,3,N)\right).
\end{align*}
Thus to prove the theorem, we first aim to show that
\begin{equation}\label{ineq1}
L(n,0,N)L(n,4,N)+3L^2(n,2,N)-4U(n,1,N)U(n,3,N)>0.
\end{equation}
Choosing $N=14$ and with \texttt{Reduce} command in Mathematica, for $n\ge 2469$, we have \eqref{ineq1}. Computing
$$\underset{s\in\{0,1,2,3,4\}}\max\{n(14,s)\}\le 5019,$$
we finally conclude that for $n\ge 5019$,
$$A\left(q(n),q(n+1), q(n+2), q(n+3), q(n+4)\right)>0.$$
For the remaining cases $229\le n\le5018$, we verified with Mathematica.\qed

{\noindent\it Proof of Theorem \ref{A-r-thm}.}
Making the shift $n\to n+1$, it is equivalent to show that for $n\ge 278$,
$$4\left(1+\frac{\pi^2}{32(n+1)^3}\right)q(n+1)q(n+3)>q(n)q(n+4)+3q(n+2)^2.$$
Since for $n\ge 1$,
$$\frac{\pi^2}{32(n+1)^3}\ge \frac{\pi^2}{32n^3}-\frac{1}{n^{\frac 72}},$$
thus it is suffices to show
$$4\left(1+\frac{\pi^2}{32n^3}-\frac{1}{n^{\frac 72}}\right)q(n+1)q(n+3)>q(n)q(n+4)+3q(n+2)^2.$$
Therefore, following \eqref{Maintheqn1}, we need to prove
\begin{align}\label{ineq2}
4\left(1+\frac{\pi^2}{32n^3}-\frac{1}{n^{\frac 72}}\right)L(n,1,N)L(N,3,N)>U(n,0,N)U(n,4,N)+3U^2(n,2,N).
\end{align}
Choosing $N=14$ and with \texttt{Reduce} command in Mathematica, for $n\ge 5885$, we have \eqref{ineq2}. Computing
$$\underset{s\in\{0,1,2,3,4\}}\max\{n(14,s)\}\le 5019,$$
we finally conclude that for $n\ge 5885$,
$$4\left(1+\frac{\pi^2}{32(n+1)^3}\right)q(n+1)q(n+3)>q(n)q(n+4)+3q(n+2)^2.$$
For the remaining cases $278\le n\le 5884$, we verified with Mathematica.\qed

{\noindent\it Proof of Theorem \ref{B-thm}.}
Making the shift $n\to n+1$, it is equivalent to show that for $n\ge 271$,
$$B\left(q(n),q(n+1), q(n+2), q(n+3), q(n+4)\right)>0.$$
Following \eqref{Maintheqn1}, we have
\begin{align*}
B\left(q(n),q(n+1), q(n+2), q(n+3), q(n+4)\right)&\\
&\hspace{-7 cm}>\left(\frac{e^{\pi\sqrt{\frac{n}{3}}}}{4\cdot3^{1/4}n^{3/4}}\right)^2\Biggl(L^3(n,2,N)+L(n,0,N)L^2(n,3,N)+L^2(n,1,N)L(n,4,N)\\
&\hspace{-5 cm}-U(n,0,N)U(n,2,N)U(n,4,N)-2U(n,1,N)U(n,2,N)U(n,3,N)\Biggr).
\end{align*}
Thus to prove the theorem, we show that
\begin{align}\label{ineq3}
L^3(n,2,N)+L(n,0,N)L^2(n,3,N)+L^2(n,1,N)L(n,4,N)&\nonumber\\
&\hspace{-8 cm}>U(n,0,N)U(n,2,N)U(n,4,N)+2U(n,1,N)U(n,2,N)U(n,3,N).
\end{align}
Choosing $N=24$ and with \texttt{Reduce} command in Mathematica, for $n\ge 9800$, we have \eqref{ineq3}. Computing
$$\underset{s\in\{0,1,2,3,4\}}\max\{n(24,s)\}\le 18502,$$
we finally conclude that for $n\ge 18502$,
$$B\left(q(n),q(n+1), q(n+2), q(n+3), q(n+4)\right)>0.$$
For the remaining cases $271\le n\le18501$, we verified with Mathematica.\qed

{\noindent\it Proof of Theorem \ref{B-r-thm}.}
Making the shift $n\to n+1$, it is equivalent to show that for $n\ge 308$,
\begin{align*}
	&\left(1+\frac{\pi^3}{288\sqrt{3}(n+1)^{\frac{9}{2}}}\right)
\left(2q(n+1)q(n+2)q(n+3)+q(n)q(n+2)q(n+4)\right)
\\
&\hspace{5cm}>q(n+2)^3+q(n)q(n+3)^2+q(n+1)^2q(n+4).
\end{align*}
Since for $n\ge 1$,
$$\frac{\pi^3}{288\sqrt{3}(n+1)^{\frac 92}}\ge \frac{\pi^3}{288\sqrt{3}n^{\frac 92}}-\frac{1}{4 n^5},$$
thus it is suffices to show
\begin{align*}
	&\left(1+\frac{\pi^3}{288\sqrt{3}n^{\frac{9}{2}}}-\frac{1}{4n^5}\right)
\left(2q(n+1)q(n+2)q(n+3)+q(n)q(n+2)q(n+4)\right)
\\
&\hspace{5cm}>q(n+2)^3+q(n)q(n+3)^2+q(n+1)^2q(n+4).
\end{align*}
Therefore, following \eqref{Maintheqn1}, we need to prove
\begin{align}\label{ineq4}
&\left(1+\frac{\pi^3}{288\sqrt{3}n^{\frac{9}{2}}}-\frac{1}{4n^5}\right)\left(L(n,0,N)L(n,2,N)L(n,4,N)+2L(n,1,N)L(n,2,N)L(n,3,N)\right)\nonumber\\
&\hspace{3 cm}>U^3(n,2,N)+U(n,0,N)U^2(n,3,N)+U^2(n,1,N)U(n,4,N).
\end{align}
Choosing $N=24$ and with \texttt{Reduce} command in Mathematica, for $n\ge 18225$, we have \eqref{ineq4}. Computing
$$\underset{s\in\{0,1,2,3,4\}}\max\{n(24,s)\}\le 18502,$$
we finally conclude that for $n\ge 18502$,
\begin{align*}
&\left(1+\frac{\pi^3}{288\sqrt{3}(n+1)^{\frac{9}{2}}}\right)
\left(2q(n+1)q(n+2)q(n+3)+q(n)q(n+2)q(n+4)\right)
\\
&\hspace{5cm}>q(n+2)^3+q(n)q(n+3)^2+q(n+1)^2q(n+4).
\end{align*}
For the remaining cases $308\le n\le 18501$, we verified with Mathematica.\qed

{\noindent\it Proof of Theorem \ref{d-turan}.}
Making the shift $n\to n+2$, it is equivalent to show that for $n\ge 271$,
\begin{align*}
\left({q(n+2)}^2-q(n+1)q(n+3)\right)^2
>\left({q(n+1)}^2-q(n)q(n+2)\right)\left({q(n+3)}^2-q(n+2)q(n+4)\right).
\end{align*}
Following \eqref{Maintheqn1}, it is enough to show that
\begin{align}\label{ineq5}
\left(L^2(n,2,N)-U(n,1,N)U(n,3,N)\right)^2&>\left(U^2(n,1,N)-L(n,0,N)L(n,2,N)\right)\nonumber\\
&\hspace{1 cm}\times\left(U^2(n,3,N)-L(n,2,N)L(n,4,N)\right).
\end{align}
Choosing $N=14$ and with \texttt{Reduce} command in Mathematica, for $n\ge 3153$, we have \eqref{ineq5}. Computing
$$\underset{s\in\{0,1,2,3,4\}}\max\{n(14,s)\}\le 5019,$$
we finally conclude that for $n\ge 5019$,
$$\left({q(n+2)}^2-q(n+1)q(n+3)\right)^2
>\left({q(n+1)}^2-q(n)q(n+2)\right)\left({q(n+3)}^2-q(n+2)q(n+4)\right).$$
For the remaining cases $271\le n\le 5018$, we verified with Mathematica.\qed

{\noindent\it Proof of Theorem \ref{c-d-turan}.}
Making the shift $n\to n+2$, it is equivalent to show that for $n\ge 344$,
\begin{align*}
&\left({q(n+1)}^2-q(n)q(n+3)\right)^2
\\
&\qquad\qquad<\left({q(n+1)}^2-q(n)q(n+2)\right)
\left({q(n+3)}^2-q(n+2)q(n+4)\right)\left(1+\frac{\pi}{2\sqrt{3}(n+1)^{\frac 32}}\right).
\end{align*}
Since for $n\ge 1$,
$$\frac{\pi}{2\sqrt{3}(n+1)^{\frac 32}}\ge \frac{\pi}{2\sqrt{3}n^{\frac 32}}-\frac{1}{ n^2},$$
thus it is suffices to show
\begin{align*}
&\left({q(n+1)}^2-q(n)q(n+3)\right)^2
\\
&\qquad\qquad<\left({q(n+1)}^2-q(n)q(n+2)\right)
\left({q(n+3)}^2-q(n+2)q(n+4)\right)\left(1+\frac{\pi}{2\sqrt{3}n^{\frac 32}}-\frac{1}{ n^2}\right).
\end{align*}
Therefore, following \eqref{Maintheqn1}, we need to prove
\begin{align}\label{ineq6}
&\left(1+\frac{\pi}{2\sqrt{3}n^{\frac 32}}-\frac{1}{ n^2}\right)\left(L^2(n,1,N)-U(n,0,N)U(n,2,N)\right)\left(L^2(n,3,N)-U(n,2,N)U(n,4,N)\right)\nonumber\\
&\hspace{7 cm}>\left(U^2(n,2,N)-L(n,1,N)L(n,3,N)\right)^2.
\end{align}
Choosing $N=14$ and with \texttt{Reduce} command in Mathematica, for $n\ge 7056$, we have \eqref{ineq6}. Computing
$$\underset{s\in\{0,1,2,3,4\}}\max\{n(14,s)\}\le 5019,$$
we finally conclude that for $n\ge 7056$,
\begin{align*}
&\left({q(n+1)}^2-q(n)q(n+3)\right)^2
\\
&\qquad\qquad<\left({q(n+1)}^2-q(n)q(n+2)\right)
\left({q(n+3)}^2-q(n+2)q(n+4)\right)\left(1+\frac{\pi}{2\sqrt{3}(n+1)^{\frac 32}}\right).
\end{align*}
For the remaining cases $344\le n\le 7055$, we verified with Mathematica.\qed

{\noindent\it Proof of Theorem \ref{3-laguerre}.}
According to \eqref{Maintheqn1}, it is sufficient to show that
\begin{align}\label{ineq-L3}
10L^2(n,3,N)+6L(n,1,N)L(n,5,N)>15U(n,2,N)U(n,4,N)+U(n,0,N)U(n,6,N).
\end{align}
Choosing $N=24$ and with \texttt{Reduce} command in Mathematica, for $n\ge 12884$, we have \eqref{ineq-L3}. Computing
$$\underset{s\in\{0,1,2,3,4,5,6\}}\max\{n(24,s)\}\le 18502,$$
we finally conclude that for $n\ge 18502$,
\begin{align*}
 10{q(n+3)}^2+6q(n+1)q(n+5)>15q(n+2)q(n+4)+q(n)q(n+6).
\end{align*}
For the remaining cases $651\le n\le 18501$, we verified with Mathematica.\qed

{\noindent\it Proof of Theorem \ref{c-3-laguerre}.}
Based on \eqref{Maintheqn1}, we need to prove
\begin{align}\label{ineq-c-L3}
&\left(1+\frac{5\pi^3}{256\sqrt{3}n^{\frac{9}{2}}}\right)
\left(15L(n,2,N)L(n,4,N)+L(n,0,N)L(n,6,N)\right)
\nonumber\\
&\hspace{7 cm}>10U^2(n,3,N)+6U(n,1,N)U(n,5,N).
\end{align}
Choosing $N=24$ and with \texttt{Reduce} command in Mathematica, for $n\ge 17880$, we have \eqref{ineq-c-L3}. Computing
$$\underset{s\in\{0,1,2,3,4,5,6\}}\max\{n(24,s)\}\le 18502,$$
we finally conclude that for $n\ge 18502$,
\begin{align*}
 10{q(n+3)}^2&+6q(n+1)q(n+5)
 \\
 &<\left(15q(n+2)q(n+4)+q(n)q(n+6)\right)
 \left(1+\frac{5\pi^3}{256\sqrt{3}n^{\frac{9}{2}}}\right).
\end{align*}
For the remaining cases $715\le n\le 18501$, we verified with Mathematica.\qed

\vspace{0.5cm}
\baselineskip 15pt
{\noindent\bf\large{\ Acknowledgements}} \vspace{7pt} \par

We wish to thank the referee for valuable suggestions. The first author would like to thank the institute for its hospitality and support. The second author is supported by the National Natural Science Foundation of China (Grant No. 12371327).

\end{document}